\documentclass[11pt]{amsart}
\usepackage{fullpage}
\usepackage{amsmath, amssymb, amsthm}
\usepackage[foot]{amsaddr}

\usepackage{color}
\usepackage[colorlinks=true,linkcolor=red,citecolor=blue,urlcolor=cyan]{hyperref}

\newtheorem{thm}{Theorem}[section]
\newtheorem{cor}[thm]{Corollary}
\newtheorem{lem}[thm]{Lemma}

\begin{document}
\title{Analyticity of Steklov eigenvalues of \\  nearly-circular and nearly-spherical domains}

\author{Robert Viator}
\address{Department of Mathematics, Southern Methodist University, Dallas, TX}
\email{rviator@ima.umn.edu}

\author{Braxton Osting}
\address{Department of Mathematics, University of Utah, Salt Lake City, UT}
\email{osting@math.utah.edu}

\subjclass[2010]{
26E05, 
35C20, 
35P05, 
41A58} 

\keywords{Dirichlet-to-Neumann operator, Steklov eigenvalues, perturbation theory}

\date{\today}

\begin{abstract} 
We consider the Dirichlet-to-Neumann operator (DNO) on nearly-circular and nearly-spherical domains in two and three dimensions, respectively. Treating such domains as perturbations of the ball, we  prove the analyticity of the DNO with respect to the domain perturbation parameter. Consequently, the Steklov eigenvalues are also shown to be analytic in the domain perturbation parameter. To obtain these results, we use the strategy of Nicholls and Nigam (2004); we transform the equation on the perturbed domain to a ball and geometrically bound the Neumann expansion of the transformed Dirichlet-to-Neumann operator. 
\end{abstract}

\maketitle

\section{Introduction}

Let $\Omega_\varepsilon \subset \mathbb R^d$ for $d=2,3$ be a nearly-circular or nearly-spherical domain of the form 
\begin{equation} \label{e:Omega}
\Omega_\varepsilon = \{ (r,\hat \theta) \colon r \leq  1 + \varepsilon \rho(\hat \theta), \ \hat \theta \in S^{d-1} \},
\end{equation} 
where the \emph{domain perturbation function} $\rho \in C^{s+2}(S^{d-1})$ for some $s\in \mathbb N $ and 
the \emph{perturbation parameter}, $\varepsilon \geq 0$, is assumed to be small in magnitude. 
We consider the Steklov eigenproblem on the perturbed domain $\Omega_\varepsilon$, 
\begin{subequations} \label{e:Steklov}
\begin{align}
\label{e:Steklova}
\Delta u_\varepsilon &= 0  && \textrm{in } \Omega_\varepsilon \\ 
\label{e:Steklovb}
\partial_{n_\varepsilon} u_\varepsilon &= \sigma_\varepsilon u_\varepsilon && \textrm{on } \partial \Omega_\varepsilon.
\end{align}
\end{subequations}
Here $\Delta$ is the Laplacian on $H^1(\Omega_\varepsilon)$ and $\partial_{n_\varepsilon} = \hat n_\varepsilon \cdot \nabla$ denotes the outward normal derivative on the boundary of $\Omega_\varepsilon$. 
It is well-known that the Steklov spectrum is discrete, real, and non-negative; we enumerate the eigenvalues in increasing order, 
$0=\sigma_0(\Omega_\varepsilon) < \sigma_1(\Omega_\varepsilon) \leq \sigma_2(\Omega_\varepsilon) \cdots \to \infty$. 
The Steklov spectrum coincides with the spectrum of the Dirichlet-to-Neumann operator (DNO), $G_\varepsilon \colon H^{s + \frac{1}{2}} (\partial \Omega_\varepsilon) \to H^{s - \frac{1}{2}} (\partial \Omega_\varepsilon)$, which maps 
$$ 
\xi \mapsto  G_\varepsilon \xi = \partial_{n_\varepsilon} u_\varepsilon, 
$$ 
where $u_\varepsilon$ is the \emph{harmonic extension} of $\xi$ to $\Omega_\varepsilon$, satisfying
\begin{subequations} \label{e:HarmExt}
\begin{align}
\label{e:HarmExta}
\Delta u_\varepsilon &= 0  && \textrm{in } \Omega_\varepsilon \\ 
\label{e:HarmExtb}
u_\varepsilon(\hat \theta) &= \xi(\hat \theta)  && \textrm{on } \partial \Omega_\varepsilon.
\end{align}
\end{subequations}
We refer the reader to  \cite{girouard2014spectral} for a general description of the Steklov spectrum.

The \emph{goal of this paper} is to prove the analyticity of the Steklov eigenvalues, $\sigma_\varepsilon$, in the perturbation parameter $\varepsilon$.
Our main result is the following theorem. 
\begin{thm} \label{t:Thm1}
Let $d=2$ or $3$ and $s \in \mathbb N$. 
If $\rho \in C^{s+2}(S^{d-1})$, then the Dirichlet-to-Neumann operator (DNO), 
$G_\varepsilon \colon H^{s+\frac 3 2}(\partial \Omega_\varepsilon ) \to H^{s + \frac 1 2}(\partial \Omega_\varepsilon )$,
is analytic in the domain parameter $\varepsilon$. 
More precisely, if $\rho \in C^{s+2}(S^{d-1})$, then there exists a Neumann series, 
$G_\varepsilon = \sum_{n=0}^\infty \varepsilon^n G_n$, 
that converges strongly as an operator from 
$H^{s+\frac 3 2}(S^{d-1})$ to $H^{s+ \frac 1 2}(S^{d-1})$. 
That is, there exists constants $K_1$ and $C$ such that 
$$
\|G_n \xi \|_{H^{s+\frac 1 2}(D)} \leq K_1 \| \xi \|_{H^{s+ \frac 3 2}(S^{d-1})} B^n
$$
for any $B > C |\rho|_{C^{s+2}}$. 
\end{thm}

We prove Theorem~\ref{t:Thm1} in two and three dimensions separately; these proofs can be found in Sections~\ref{s:DNO-anal-2D} and \ref{s:DNO-anal-3D}, respectively. 
In both dimensions, our proof of Theorem~\ref{t:Thm1} follows the strategy in \cite{Nicholls_2004}. We first show the analyticity of the harmonic extension, that is, for fixed $\xi(\hat \theta)$ the solution $u_\varepsilon$ in \eqref{e:HarmExt} is analytic in $\varepsilon$. Using this, we then prove that the DNO, $G_\varepsilon$, is also analytically dependent on $\varepsilon$, establishing Theorem~\ref{t:Thm1}.  

Using an analyticity result in \cite{Kato_1966}, we obtain the analytic dependence of the Steklov eigenvalues $\{ \sigma_j(\varepsilon) \}_{j\in \mathbb{N}}$ on $\varepsilon$ within the same disc of convergence as in Theorem~\ref{t:Thm1}, as stated in the following corollary. 

\begin{cor} \label{c:Cor1}
The Steklov eigenvalues, $\sigma_\varepsilon$, consist of branches of one or several analytic functions which have at most algebraic singularities near $\varepsilon=0$.  The same is true of the corresponding eigenprojections. 
\end{cor}

The proof of Corollary~\ref{c:Cor1} is given in Section~\ref{s:Steklov-anal}. 

Corollary~\ref{c:Cor1} justifies Assumption 1.1 in  \cite{Viator_2018}. Here, the first two terms of the asymptotoic series for $\sigma_\varepsilon$ are computed for reflection-symmetric nearly-circular domains.  
Corollary~\ref{c:Cor1} also justifies the computation of the shape derivative that appears in \cite{Akhmetgaliyev2016}. Here, numerical methods are developed for the eigenvalue optimization problem of maximizing the $k$-th Steklov eigenvalue as a function of the domain with a volume constraint.

\section{Two-dimensional nearly-circular domains} 
Here we consider the Steklov eigenproblem \eqref{e:Steklov} in $\mathbb{R}^2$.  We will identify $\hat \theta$ with its corresponding angle $\theta$ made with the positive $x$-axis, as usual. We write the Fourier series for $f\colon S^1 \to \mathbb C$ as 
\[
f(\theta) = \frac{1}{\sqrt{2 \pi}} \sum_{k \in \mathbb Z} \hat{f}(k) e^{i k \theta}, 
\qquad \textrm{where} \ \ 
\hat{f}(k) = \frac{1}{\sqrt{2 \pi}} \int_0^{2 \pi} f(\theta) e^{- i k \theta } \ d \theta.
\]

Denoting $\langle k \rangle = \sqrt{1 + k^2}$, we introduce the spaces $L^2(S^1)$ and $H^1(S^1)$ with norms
\begin{align*}
\|f\|^2_{L^2(S^1)} &= \int_0^{2\pi} |f|^2 \ d \theta = \sum_{k \in \mathbb Z} |\hat{f}(k) |^2\\
\| f \|^2_{H^1(S^1)} &= \int_0^{2 \pi} |f(\theta)|^2 + |f'(\theta)|^2 \ d \theta 
 = \sum_{k \in \mathbb Z} | \hat{f}(k) |^2 + k^2 | \hat{f}(k) |^2 
= \sum_{k \in \mathbb Z} \langle k\rangle^2 | \hat{f}(k) |^2.
\end{align*}
Similarly, we define the $H^s(S^1)$ space with norm 
$ \| f \|^2_{H^s(S^1)}  = \sum_{k \in \mathbb Z} \langle k\rangle^{2s} | \hat{f}(k) |^2$. 

\subsection{Analyticity of the harmonic extension for nearly-circular domains} 
\label{s:AnalHarmExt2d}
We first consider the problem  of harmonically extending a function $\xi(\theta)$ from $\partial \Omega_\varepsilon$ to $\Omega_\varepsilon$, 
\begin{subequations} \label{e:Laplace2d}
\begin{align}
& \left[r^{-1} \partial_r r \partial_r  + r^{-2} \partial_\theta^2 \right] v = 0  \\
& v(1+\varepsilon \rho(\theta), \theta) = \xi(\theta). 
\end{align}
\end{subequations}
Mapping $\Omega_\varepsilon$ to the unit disk, $D = \Omega_0$, we make the change of variables 
\begin{align} \label{e:ChangeVariables2d}
(r',\theta') = \left( (1+\varepsilon \rho(\theta))^{-1}  r, \theta \right). 
\end{align}
The partial derivatives in the new coordinates are given by
\[
\frac{\partial}{\partial r} = \frac{1}{1+\varepsilon \rho(\theta')} \frac{\partial}{ \partial r'}  
\qquad \textrm{and} \qquad 
\frac{\partial}{\partial \theta} = \frac{\partial }{ \partial \theta' } - \frac{\varepsilon r' \rho'(\theta')}{1+ \varepsilon \rho(\theta')} \frac{\partial }{ \partial r'}. 
\]
Applying this change of coordinates to the Laplace equation \eqref{e:Laplace2d} and setting \[u_\varepsilon(r',\theta') = v( (1+ \varepsilon \rho(\theta') ) r', \theta' ),\] we obtain the problem 
\begin{align*}
 \frac{1}{r'  \left( 1 + \varepsilon \rho(\theta') \right)^2} \frac{\partial}{ \partial r'} r' \frac{\partial u_\varepsilon }{ \partial r'} +   \frac{1}{(r')^2  \left( 1 + \varepsilon \rho(\theta') \right)^2} \left(   
  \frac{\partial }{ \partial \theta' } - \frac{\varepsilon r' \rho'(\theta')}{1+ \varepsilon \rho(\theta')} \frac{\partial }{ \partial r'}   \right)^2 u_\varepsilon   = 0. 
\end{align*}
Multiplying both sides by $\left(1+ \varepsilon \rho(\theta') \right)^2$ and dropping the primes on the transformed variables yields 
\[  r^{-1} \partial_r r \partial_r u_\varepsilon +  r^{-2}  \left(   \partial_\theta - \varepsilon r (1+ \varepsilon \rho(\theta) )^{-1} \rho'(\theta) \partial_r   \right)^2 u_\varepsilon   = 0. \]
Expanding the operator in the second term on the left hand side, we obtain
\begin{align*}
&\left( r^{-1} \partial_r r \partial_r  + r^{-2} \partial_\theta^2 \right) u_\varepsilon 
  -   \varepsilon r^{-1} \left( 1 + \varepsilon\rho(\theta) \right)^{-1}   \left( 2 \rho'(\theta)  \partial_\theta + \rho''(\theta) \right)  \partial_r u_\varepsilon  \\
 \nonumber
& \qquad \qquad \qquad \qquad \qquad  
+ \varepsilon^2 r^{-1} \left( 1 + \varepsilon\rho(\theta) \right)^{-2}  \left(\rho'(\theta) \right)^2  \partial_r \left( 2 + r \partial_r \right) u_\varepsilon   = 0.
\end{align*}
Again multiplying both sides by $\left(1+ \varepsilon \rho(\theta) \right)^2$, we obtain 
the transformed Laplace equation, 
\begin{subequations} \label{e:Laplace-transformed}
\begin{align}
\label{e:Laplace-transformed-a}
 & \Delta u_\varepsilon 
   =    \varepsilon  L_1 u_\varepsilon +  \varepsilon^2 L_2 u_\varepsilon \\
& u_\varepsilon(1,\theta) = \xi(\theta) . 
\end{align}
\end{subequations}
where
\begin{align*}
\Delta & = \left( r^{-1} \partial_r r \partial_r  + r^{-2} \partial_\theta^2 \right) \\
L_1 & = 2\rho'(\theta)r^{-1}\partial_\theta \partial_r   
  + \rho''(\theta) r^{-1}\partial_r  - 2\rho(\theta) \left[ \partial_r^2 + r^{-1} \partial_r  + r^{-2}\partial_\theta^2 \right]  \\ 
L_2 & = 2 \rho(\theta) \rho'(\theta) r^{-1} \partial_\theta \partial_r + \rho(\theta) \rho''(\theta) r^{-1} \partial_r   - (\rho'(\theta))^2 \partial_r^2 \\
& \qquad -2 (\rho'(\theta))^2 r^{-1} \partial_r  
- \rho^2(\theta) \left[ \partial_r^2 + r^{-1} \partial_r   + r^{-2} \partial_\theta^2 \right].   
\end{align*}

We formally expand the solution, $u_\varepsilon$, in powers of $\varepsilon$, 
\begin{equation} \label{e:u-expansion} 
u_\varepsilon(r,\theta) = \sum_{n=0}^\infty \varepsilon^n u_n(r,\theta). 
\end{equation}
Next, we collect terms in powers of  $\varepsilon$. At $O(\varepsilon^0)$, we obtain 
\begin{align*}
&  \Delta  u_0(r,\theta) = 0  \\
& u_0(1, \theta) = \xi(\theta). 
\end{align*}
At $O(\varepsilon^n)$ for $n> 0$, we obtain 
\begin{align*}
& \Delta  u_n(r,\theta) 
=  L_1 u_{n-1} + L_2 u_{n-2}  \\
& u_n(1, \theta) = 0. 
\end{align*}

We next show that there exists a unique solution of \eqref{e:Laplace-transformed} of the form in \eqref{e:u-expansion}. 
The following Lemma is analogous to \cite[Lemma 4]{Nicholls_2004}.

\begin{lem}
\label{lem21}
For $s\in \mathbb N$, there is a constant $K_0 > 0$ such that for any $F \in H^{s-1}(D)$ and $\xi \in H^{s+ \frac 1 2}(S^1)$, the solution of 
\begin{align*}
    &\Delta w(r,\theta) = F(r,\theta) && (r,\theta) \in D \\
    &w(1,\theta) = \xi(\theta) && \theta \in S^1
\end{align*}
satisfies
\[
\| w \|_{H^{s+1}(D)} \leq K_0 \left( \| F \|_{H^{s-1}(D)}  + \| \xi \|_{H^{s+ \frac 1 2}(S^1)} \right). 
\]
\end{lem}
\begin{proof}
We will prove the result for $s=0$. Since $\xi \in H^{\frac 1 2 }(S^1)$, we have the Fourier series
\[
\xi(\theta) = \frac{1}{\sqrt{2 \pi}} \sum_{k \in \mathbb Z} \hat{\xi}(k) e^{i k \theta}, 
\qquad \textrm{where} \ \ 
\hat{\xi}(k) = \frac{1}{\sqrt{2 \pi}} \int_0^{2 \pi} \xi(\theta) e^{- i k \theta } \ d \theta.
\]
Setting $v = w - \Phi$, where 
$\Phi(r,\theta) = \frac{1}{\sqrt{2 \pi}} \sum_{k \in \mathbb Z} \hat{\xi}(k) r^{|k|} e^{i k \theta}$,
we have that 
\begin{align*}
&\Delta \Phi(r,\theta) = 0 && (r,\theta) \in D  \\
&\Phi(1,\theta) = \xi(\theta) 
\end{align*} 
and 
\begin{subequations} \label{e:DirPartw}
\begin{align}
    \label{e:DirPartw-a}
    &\Delta v(r,\theta) = F(r,\theta) && (r,\theta) \in D \\
    \label{e:DirPartw-b}
    &v(1,\theta) = 0. 
\end{align}
\end{subequations}
Using $\hat\Phi(r,k) = \hat \xi (k) r^{|k|}$, a straightforward integration yields
\begin{subequations} \label{e:bndPhi2d}
\begin{align}
\| \Phi \|^2_{H^1(D)} 
&= \sum_{k\in \mathbb Z} \int_0^1 \left[ \left( 1 + r^{-2} |k|^2 \right) |\hat \Phi(r,k)|^2 + | \partial_r \hat \Phi(r,k) |^2\right] \ r dr \\
&= \sum_{k\in \mathbb Z} |\hat \xi (k)|^2 \int_0^1 \left[ \left( 1 + 2 r^{-2} |k|^2 \right)  r^{2 |k|} \right]  \ r dr \\
&\leq C_{\frac 1 2}^2 \sum_{k\in \mathbb Z} \langle k \rangle |\hat \xi (k)|^2\\ 
&= C_{\frac 1 2}^2 \|\xi\|^2_{H^{\frac 1 2}(S^1)},
\end{align}
\end{subequations}
for some constant $C_{\frac 1 2} > 0$. 

Multiplying \eqref{e:DirPartw-a} by $\overline v$, integrating by parts, and using \eqref{e:DirPartw-b} yields
\[ 
\| v \|_{H_0^1(D)}^2 = \int_D \nabla v \cdot \nabla \overline{v} = - \int_D F \overline{v} .
\]
By the duality of $H_0^1(D)$ and $H^{-1}(D)$, we have
$\int_D | F \overline{v}|  \leq \| F \|_{H^{-1}(D)} \| v \|_{H_0^1(D)}$. 
Since $v \in H_0^1(D)$, by the Poincar\'e inequality, there exists a constant $C_D$ such that $ \| v \|_{H^1(D)} \leq  C_D \| v \|_{H^1_0(D)}$ and we conclude that 
\begin{equation} \label{e:bndv2d}
\| v \|_{H^1(D)} \leq  C_D \| v \|_{H_0^1(D)} \leq C_D \| F \|_{H^{-1}(D)}.
\end{equation}

Using the decomposition $w = v + \Phi$ and using \eqref{e:bndPhi2d} and \eqref{e:bndv2d}, we obtain
\[ 
\| w \|_{H^1(D)} \leq \| v \|_{H^1(D)} + \| \Phi \|_{H^1(D)} 
\leq C_D \| F \|_{H^{-1}(D)} + C_{\frac 1 2} \| \xi \|_{H^{\frac 1 2}(S^1)}.
\]
Taking $K_0 = \max\{ C_{\frac 1 2}, C_D \}$ yields the desired result for $s=0$. The proof for $s\geq 1$ is similar. 
\end{proof} 

The next Lemma  will be used to prove the inductive step in the proof of Theorem \ref{e:Thm23} and is analogous to \cite[Lemma 5]{Nicholls_2004}. 
In the proof, we use the following result \cite{nicholls2001new,Nicholls_2004}. 
For 
$\varepsilon \geq 0$, 
$s\in \mathbb N$, 
$f \in C^s(S^{d-1})$, 
$u \in H^s(B^d)$, 
$g\in C^{s + \frac 1 2 + \delta}(S^{d-1})$, and 
$\mu \in H^{s + \frac 1 2}(S^{d-1})$, 
there exists a constant $M = M(s,d)$ so that  
\begin{subequations} \label{e:Nicholls}
\begin{align}
    \| f u \|_{H^s} &\leq M(s,d) \ |f|_{C^s} \ \| u \|_{H^s} \\
    \| g \mu \|_{H^{s+\frac 1 2}} &\leq M(s,d) \ |g|_{C^{s+\frac 1 2 + \delta}}  \ \| \mu \|_{H^{s+ \frac 1 2}}.
\end{align}
\end{subequations}

\begin{lem}
\label{lem22}
Let $s\in \mathbb N$ and let $\rho \in C^{s+2}(S^1)$.  Assume that $K_1$ and $B$ are constants so that 
$$
\| u_n \|_{H^{s+2}(D)} \leq K_1 B^n
\qquad \qquad \textrm{for all } \ \  n < N.
$$
If $B >|\rho|_{C^{s+2}}$, 
then there exists a $C_0$ such that 
\begin{align*}
\| L_1 u_{N-1} \|_{H^s(D)} &\leq K_1 |\rho|_{C^{s+2}} C_0 B^{N-1} \\
\| L_2 u_{N-2} \|_{H^s(D)} &\leq K_1 |\rho|_{C^{s+2}} C_0 B^{N-1}.
\end{align*}
\end{lem}

\begin{proof}
We begin by rewriting
\begin{align*}
L_1 u_{N-1} = & \ 2\rho'(\theta)r^{-1}\partial_\theta \partial_r u_{N-1}  
  - \rho''(\theta) r^{-2}\partial_{\theta}^2 u_{N-1} - \rho''(\theta) \partial_r^2 u_{N-1} \\
  & \ + (\rho''(\theta) - 2\rho(\theta)) \left[ \partial_r^2 + r^{-1} \partial_r  + r^{-2}\partial_\theta^2 \right]u_{N-1}
\end{align*}
as well as
\begin{align*}
L_2 u_{N-2} = & 2 \rho(\theta) \rho'(\theta) r^{-1} \partial_\theta \partial_r u_{N-2} - (\rho(\theta) \rho''(\theta) -2 (\rho'(\theta))^2) \left[r^{-2} \partial_\theta^2 + \partial_r^2 \right] u_{N-2}   - (\rho'(\theta))^2 \partial_r^2 u_{N-2} \\
&  
+ (\rho(\theta) \rho''(\theta) -2 (\rho'(\theta))^2 - \rho^2(\theta)) \left[ \partial_r^2 + r^{-1} \partial_r   + r^{-2} \partial_\theta^2 \right] u_{N-2}.
\end{align*}

First, we measure $L_1 u_{N-1}$ in $H^{s}(D)$ and use the triangle inequality and \eqref{e:Nicholls} to obtain:
\begin{align*}
\| L_1 u_{N-1} \|_{H^s} & \leq \| 2\rho'(\theta)r^{-1}\partial_\theta \partial_r u_{N-1} \|_{H^s}  
+ \| \rho''(\theta) r^{-2}\partial_{\theta}^2 u_{N-1} \|_{H^s} + \| \rho''(\theta) \partial_r^2 u_{N-1}\|_{H^s} \\
& \quad + \|(\rho''(\theta) - 2\rho(\theta)) \Delta u_{N-1}\|_{H^s} \\
& \leq 2M(s)|\rho|_{C^{s+1}} \|r^{-1}\partial_\theta \partial_r u_{N-1} \|_{H^s}  
+ M(s)|\rho|_{C^{s+2}} \| r^{-2}\partial_{\theta}^2 u_{N-1} \|_{H^s} +  M(s)|\rho|_{C^{s+2}}\| \partial_r^2 u_{N-1}\|_{H^s} \\
& \quad + M(s)|\rho|_{C^{s+2}}\| \Delta u_{N-1}\|_{H^s} + 2M(s)|\rho|_{C^s} \| \Delta u_{N-1}\|_{H^s} \\
& \leq 2M(s)|\rho|_{C^{s+1}} \| u_{N-1} \|_{H^{s+2}}  
+ M(s)|\rho|_{C^{s+2}} \| u_{N-1} \|_{H^{s+2}} +  M(s)|\rho|_{C^{s+2}}\| u_{N-1}\|_{H^{s+2}} \\
& \quad + M(s)|\rho|_{C^{s+2}}\|u_{N-1}\|_{H^{s+2}} + 2M(s)|\rho|_{C^s} \|u_{N-1}\|_{H^{s+2}}\\
& \leq K_1 |\rho|_{C^{s+2}} C_0 B^{N-1}.
\end{align*}
Here, in the third inequality, we have used that all operators acting on $u_{N-1}$ are second order. 
Similarly, we estimate $L_2 u_{N-2}$ in $H^s(D)$:
\begin{align*}
\| L_2 u_{N-2} \|_{H^s} 
& \leq 2 \|\rho \rho' r^{-1} \partial_\theta \partial_r u_{N-2}\|_{H^s} + \| (\rho \rho'' -2 (\rho')^2) \left[r^{-2} \partial_\theta^2 + \partial_r^2 \right] u_{N-2}\|_{H^s}   + \| (\rho')^2 \partial_r^2 u_{N-2}\|_{H^s} \\
& \quad + \left\| \left(\rho \rho'' -2 (\rho')^2 - \rho^2 \right) \Delta u_{n-2} \right\|_{H^s} \\
& \leq 2 M(s)|\rho|_{C^s} |\rho|_{C^{s+1}} \| r^{-1} \partial_\theta \partial_r u_{N-2}\|_{H^s} + M(s)\left(|\rho|_{C^s} |\rho|_{C^{s+2}} + 2|\rho|^2_{C^{s+1}}\right) \|r^{-2}\partial_\theta^2 u_{N-2}\|_{H^s} \\
& \quad + M(s) \left(|\rho|_{C^s} |\rho|_{C^{s+2}} + |\rho|^2_{C^{s+1}}\right) \|\partial_r^2 u_{N-2}\|_{H^s}\\
& \quad + M(s) \left(|\rho|_{C^s} |\rho|_{C^{s+2}} + 2|\rho|^2_{C^{s+1}} + |\rho|^2_{C^s} \right) \| \Delta u_{N-2}\|_{H^s} \\
& \leq 2 M(s)|\rho|_{C^s} |\rho|_{C^{s+1}} \| u_{N-2}\|_{H^{s+2}} + M(s)\left(|\rho|_{C^s} |\rho|_{C^{s+2}} + 2|\rho|^2_{C^{s+1}}\right) \| u_{N-2}\|_{H^{s+2}} \\
& \quad + M(s) \left(|\rho|_{C^s} |\rho|_{C^{s+2}} + |\rho|^2_{C^{s+1}} \right) \| u_{N-2}\|_{H^{s+2}}\\
& \quad + M(s) \left(|\rho|_{C^s} |\rho|_{C^{s+2}} + 2|\rho|^2_{C^{s+1}} + |\rho|^2_{C^s} \right) \| u_{N-2}\|_{H^{s+2}} \\
& \leq K_1 |\rho|_{C^{s+2}}C_0 B^{N-1}.
\end{align*}
\end{proof}

The following Theorem justifies the convergnece of \eqref{e:u-expansion} for sufficiently small $\varepsilon >0$ and is analogous to \cite[Theorem 3]{Nicholls_2004}.
\begin{thm} \label{e:Thm23}
Given $s \in \mathbb N$, if $\rho \in C^{s+2}(S^1)$ and $\xi \in H^{s+\frac 3 2}(S^1)$, there exists constants $C_0$ and $K_0$ and a unique solution of \eqref{e:Laplace-transformed} such that
\begin{equation}
\label{u-bound2d}
\|u_n \|_{H^{s+2}(D)} \leq K_0 \| \xi \|_{H^{s+ \frac 3 2}(S^1)} B^n
\end{equation}
for any $B > 2 K_0 C_0 | \rho|_{C^{s+2}}$. 
\end{thm}
\begin{proof}
We proceed by induction.  For $n=0$, we use Lemma \ref{lem21} to we see
\begin{equation*} 
\| u_0 \|_{H^{s+2}} \leq K_0 \| \xi \|_{H^{s+ \frac 3 2}} B^0,
\end{equation*}
as desired to show \eqref{u-bound2d}. We now define $K_1=K_0\| \xi \|_{H^{s+ \frac{3}{2}}(S^1)}$ for the remainder of the proof to be used in Lemma \ref{lem22}. 

Suppose inequality \eqref{u-bound2d} holds for $n<N$.  Then by Lemma \ref{lem21},
\begin{equation*}
    \| u_N \|_{H^{s+2}} \leq K_0 \left( \| L_1 u_{N-1}\|_{H^s} + \| L_2 u_{N-2}\|_{H^s} \right).
\end{equation*}
By Lemma \ref{lem22}, we may bound $\|L_1 u_{N-1}\|_{H^s}$ and $\|L_2 u_{N-2}\|_{H^s}$ so that
\begin{align*}
    \| u_N \|_{H^{s+2}} & \leq 2K_0 K_1 C_0 |\rho|_{C^{s+2}} B^{N-1} \\
    & = 2 K_0^2 \| \xi \|_{H^{s+ \frac 3 2}(S^1)}C_0 |\rho|_{C^{s+2}} B^{N-1} \\
    & \leq K_0 \| \xi \|_{H^{s+ \frac 3 2}(S^1)} B^N
\end{align*}
provided $B> 2K_0 C_0 |\rho|_{C^{s+2}}$.
\end{proof}

\subsection{Proof of Theorem~\ref{t:Thm1} in two dimensions: Analyticity of the Dirichlet to Neumann operator} 
\label{s:DNO-anal-2D}

The Dirichlet to Neumann operator (DNO), $G \colon H^{s+1}(\partial \Omega_\varepsilon ) \to H^{s}(\partial \Omega_\varepsilon )$, is given by 
\begin{align*}
G \xi &= \left[ 1 + 2 \varepsilon \rho(\theta) + \varepsilon^2 \left(\rho^2(\theta) + (\rho'(\theta) \right)^2 \right]^{-\frac{1}{2}} 
\left[ \left(1+\varepsilon \rho(\theta) \right) \frac{\partial v}{\partial r}  - \frac{\varepsilon \rho'(\theta)}{1+\varepsilon \rho(\theta)}  \frac{\partial v}{\partial \theta} \right], 
\end{align*}
where $v$ is the harmonic extension of $\xi$ from $\partial \Omega_\varepsilon $ to $\Omega_\varepsilon $, satisfying \eqref{e:Laplace2d}.
Making the change of coordinates given in \eqref{e:ChangeVariables2d}, we obtain the transformed DNO, 
$G(1+ \varepsilon  \rho) \colon H^{s+1}(S^1) \to H^{s}(S^1)$, given by 
\begin{align*}
G\left(1+\varepsilon \rho(\theta') \right) \xi 
&= \left[ 1 + 2 \varepsilon \rho(\theta') + \varepsilon^2 \left(\rho^2(\theta') + (\rho'(\theta') )^2 \right) \right]^{-\frac{1}{2}} 
\left[ \left(1 + \frac{\varepsilon^2 (\rho'(\theta'))^2}{\left(1+ \varepsilon \rho(\theta') \right)^2} r'   \right)   \frac{\partial u_\varepsilon}{\partial r'}
- \frac{\varepsilon \rho'(\theta')}{1+\varepsilon \rho(\theta')}  \frac{\partial u_\varepsilon}{\partial \theta'} \right] \\
&= M_\rho(\varepsilon) \hat{G}_{\rho, \varepsilon}\xi ,
\end{align*}
where $u_\varepsilon$ satisfies \eqref{e:Laplace-transformed} and
\begin{align*}
M_\rho(\varepsilon ) & = \left[ 1 + 2 \varepsilon \rho(\theta') + \varepsilon^2 \left(\rho^2(\theta') + (\rho'(\theta') \right)^2 \right]^{-\frac{1}{2}}, \\
\hat{G}_{\rho, \varepsilon}\xi &= \left[ \left(1 + \frac{\varepsilon^2 (\rho'(\theta'))^2}{\left(1+ \varepsilon \rho(\theta') \right)^2} r'   \right)   \frac{\partial u_\varepsilon}{\partial r'}
- \frac{\varepsilon \rho'(\theta')}{1+\varepsilon \rho(\theta')}  \frac{\partial u_\varepsilon}{\partial \theta'} \right].
\end{align*}
Since $M_\rho(\varepsilon)$ is clearly analytic in $\varepsilon$, we need only show the analyticity of $\hat{G}(1+\varepsilon \rho(\theta'))$.  Dropping the prime notation on the new variables, we obtain
\begin{align*}
\left(1+ \varepsilon \rho(\theta) \right)^2 \hat{G}_{\rho, \varepsilon} \xi 
&= \left[ \left( \left(1+ \varepsilon \rho(\theta) \right)^2 + \varepsilon^2 (\rho'(\theta))^2   \right)   \partial_r u - \varepsilon \left(1+ \varepsilon \rho(\theta) \right) \rho'(\theta)  \partial_\theta u  \right]. \\
\end{align*}
We expand the non-normalized DNO, $\hat{G}_{\rho, \varepsilon}$, as a power series in $\varepsilon$
\begin{align}
\label{DNOanalytic}
\hat{G}_{\rho, \varepsilon}\xi = \sum \limits_{n=0}^\infty \varepsilon^n \hat{G}_{n,\rho}\xi, 
\end{align}
which yields the following recursive formula:
\begin{align*}
    \hat{G}_{\rho, n}\xi &= \partial_r u_n + 2\rho \partial_r u_{n-1} + \left( (\rho')^2 + \rho^2 \right) \partial_r u_{n-2} - \rho' \partial_\theta u_{n-1} - \rho \rho' \partial_\theta u_{n-2} -2\rho \hat{G}_{\rho, n-1}\xi - \rho^2 \hat{G}_{\rho, n-2}\xi.
\end{align*}
We now prove the following theorem, which proves Theorem \ref{t:Thm1} and guarantees the uniform convergence of the series \eqref{DNOanalytic} for suitably small $\varepsilon$.
\begin{thm} \label{t:ConvDNO2d}
Let $\xi \in H^{s+3/2}(S^1)$.  Then
\begin{equation}
\label{DNOrecursivebnd2d}
\|\hat{G}_{\rho, n}\xi \|_{H^{s+1/2}(S^1)} \leq K_1 \| \xi \|_{H^{s+3/2}(S^1)}B^n
\end{equation}
for $B> C |\rho|_{C^{s+2}}$. 
\end{thm}
\begin{proof}
We will proceed via induction. First, we show \eqref{DNOrecursivebnd2d} fo $n=0$:
\begin{align*}
\|\hat{G}_{\rho, 0}\xi \|_{H^{s+1/2}(S^1)} &  \leq \| \partial_r u_0\|_{H^{s+1/2}(S^1)} \leq C_1 \| \partial_r u_0\|_{H^{s+1}(D)} \\
& \leq C_1\|u_0\|_{H^{s+2}(D)} \leq C_1 K_0 \| \xi \|_{H^{s+3/2}(S^1)}. 
\end{align*}
In the second inequality of the first line, we have used the trace theorem, while Theorem \ref{e:Thm23} is used in the second line.
Now suppose that \eqref{DNOrecursivebnd2d} holds for $n<N$.  Then we have the following estimate:
\begin{align*}
\|\hat{G}_{\rho, N}\xi \|_{H^{s+1/2}(S^1)} 
& \leq \| \partial_r u_N\|_{H^{s+1/2}(S^1)} + 2\| \rho \partial_r u_{N-1}\|_{H^{s+1/2}(S^1)} + \|(\rho')^2 \partial_r u_{N-2}\|_{H^{s+1/2}(S^1)} \\
& \quad +\|\rho^2 \partial_r u_{N-2}\|_{H^{s+1/2}(S^1)} + \|\rho' \partial_\theta u_{N-1}\|_{H^{s+1/2}(S^1)} + \|\rho \rho' \partial_\theta u_{N-2}\|_{H^{s+1/2}(S^1)} \\
& \quad + 2\|\rho \hat{G}_{\rho, N-1} \xi\|_{H^{s+1/2}(S^1)} + \|\rho^2 \hat{G}_{\rho, N-2}\xi\|_{H^{s+1/2}(S^1)} \\
& \leq C_1 K_0\| \xi \|_{H^{s+3/2}(S^1)}B^N + 2 |\rho|_{C^{s+1/2 +\delta}}C_1 K_0\| \xi \|_{H^{s+3/2}(S^1)}B^{N-1} + \\
& \quad + |\rho|^2_{C^{s+3/2+\delta}}C_1 K_0\| \xi \|_{H^{s+3/2}(S^1)}B^{N-2} \\
& \quad +|\rho|^2_{C^{s+1/2+\delta}}C_1 K_0\| \xi \|_{H^{s+3/2}(S^1)}B^{N-2} +|\rho|_{C^{s+3/2+\delta}}C_1 K_0\| \xi \|_{H^{s+3/2}(S^1)}B^{N-1} \\
& \quad + |\rho|_{C^{s+1/2+\delta}}|\rho|_{C^{s+3/2+\delta}}C_1 K_0\| \xi \|_{H^{s+3/2}(S^1)}B^{N-2} \\
& \quad + 2|\rho|_{C^{s+1/2+\delta}} \|\hat{G}_{\rho, N-1} \xi\|_{H^{s+1/2}(S^1)} + |\rho|^2_{C^{s+1/2+\delta}} \| \hat{G}_{\rho, N-2}\xi\|_{H^{s+1/2}(S^1)} \\
& \leq K_1 \| \xi \|_{H^{s+3/2}(S^1)}B^N,
\end{align*}
for $B>C|\rho|_{C^{s+2}}$, where $C$ is independent of $u$, $N$, $\xi$, and $\rho$.  Here we have used the second inequality in \eqref{e:Nicholls}, as well as the trace theorem, Theorem \ref{e:Thm23}, and the inductive hypothesis on $\hat{G}_{\rho, N-2}$ and $\hat{G}_{\rho, N-2}$.
\end{proof}

\section{Three dimensional nearly-spherical domains} 
Here we consider \eqref{e:Steklov} in dimension $d=3$. We identify $\hat \theta \in \mathbb S^2$ with the inclination, $\theta\in[0,\pi]$, and azimuth, $\phi\in[0,2\pi]$. 
Let $\Omega_\varepsilon$ be a nearly-spherical domain where the perturbation function is expanded in the basis of real spherical harmonics, 
\begin{equation} \label{e:dom}
\Omega_\varepsilon=\{(r,\theta,\phi)\colon  0\le r\le 1 + \varepsilon \rho(\theta,\phi)\}, 
\qquad\text{where} \ \ \rho(\theta,\phi)=\sum_{\ell=0}^{\infty}\sum_{m = -\ell}^\ell A_{\ell,m} Y_{\ell,m}(\theta,\phi). 
\end{equation}
Here, $Y_{\ell,m}$ denote the real spherical harmonics, which are obtained from the complex spherical harmonics as follows. Define the \emph{complex spherical harmonic} by 
 \begin{equation} \label{e:ComplexHarmonic}
 Y_{\ell}^{m}(\theta,\phi)  =  \sqrt{\frac{(2\ell+1)}{4\pi}\frac{(\ell-m)!}{(\ell+m)!}}P_{\ell}^{m}(\cos(\theta))e^{im\phi}, 
 \qquad \qquad 
 \ell \geq 0, \ \  |m| \leq \ell,  
\end{equation}
where $P_{\ell}^{m}$ is the \emph{associated Legendre polynomial}, which can be defined through the Rodrigues formula, 
$P_{\ell}^{m}(x) = \frac{(-1)^{m}}{2^{\ell}\ell!}(1-x^{2})^{\frac{m}{2}}\frac{d^{m+\ell}}{dx^{m+\ell}}(x^{2}-1)^{\ell}$.
For $\ell \geq 0$ and $|m| \leq \ell$, the  \emph{real spherical harmonics} are then defined by
\begin{align}
\label{e:RealHarmonicsa}
Y_{\ell,m}(\theta,\phi) &=
\begin{cases}
\frac{i}{\sqrt{2}} \left[ Y_{\ell}^m(\theta,\phi) - (-1)^mY_{\ell}^{-m}(\theta,\phi) \right] & \text{if}\ m<0 \\
Y_{\ell}^0(\theta,\phi) & \text{if}\ m=0  \\
\frac{1}{\sqrt{2}} \left[ Y_{\ell}^{-m}(\theta,\phi) + (-1)^mY_{\ell}^{m}(\theta,\phi) \right] & \text{if}\ m>0 
\end{cases}. 
\end{align}

\subsection{Analyticity of the harmonic extension for nearly-spherical domains} 
As in Section~\ref{s:AnalHarmExt2d}, we first consider the problem  of harmonically extending a function $\xi(\theta, \phi)$ from $\partial \Omega_\varepsilon$ to $\Omega_\varepsilon$, 
\begin{subequations} \label{e:Laplace3d}
\begin{align}
\label{e:Laplace3da}
& \Delta v =  r^{-2} \partial_r \left( r^2 \partial_r v \right)  + r^{-2} \sin^{-1}(\theta)  \partial_\theta\left( \sin(\theta)\partial_\theta v \right) + r^{-2} \sin^{-2} (\theta) \partial_\phi^2 v  = 0  \\
& v(1+\varepsilon \rho(\theta, \phi), \theta, \phi) = \xi(\theta, \phi). 
\end{align}
\end{subequations}
Mapping $\Omega_\varepsilon$ to the unit ball, $B = \Omega_0$, we make the change of variables 
\begin{align} \label{e:ChangeVariables3d}
(r',\theta', \phi') = \left( (1+\varepsilon \rho(\theta, \phi))^{-1}  r, \theta, \phi \right). 
\end{align}
The partial derivatives in the new coordinates are given by
\begin{align*}
\frac{\partial}{\partial r} &= \frac{1}{1+\varepsilon \rho(\theta', \phi')} \frac{\partial}{ \partial r'}  \\
\frac{\partial}{\partial \theta} &= \frac{\partial }{ \partial \theta' } - \frac{\varepsilon r' \rho_\theta(\theta', \phi')}{1+ \varepsilon \rho(\theta',\phi')} \frac{\partial }{ \partial r'} \\
\frac{\partial}{\partial \phi} &= \frac{\partial }{ \partial \phi' } - \frac{\varepsilon r' \rho_\phi(\theta', \phi')}{1+ \varepsilon \rho(\theta',\phi')} \frac{\partial }{ \partial r'}.
\end{align*}
Applying this change of coordinates to the Laplace equation \eqref{e:Laplace3da}, setting 
\[u_\varepsilon(r',\theta', \phi') = v( (1+ \varepsilon \rho(\theta') ) r', \theta', \phi' ),\] 
multiplying by $(1+ \varepsilon \rho )^4$,
and dropping the primes on the transformed variables yields  
\begin{align*}
(1+ \varepsilon \rho )^4 \Delta v &= 
(1+ \varepsilon \rho )^2 \Delta u_{\varepsilon} - \varepsilon (1+ \varepsilon \rho) \Big[ \sin^{-1}(\theta) \partial_\theta \Big( \sin(\theta) \partial_\theta \rho \Big) + \sin^{-2}(\theta) \partial_\phi^2 \rho   \Big] (r^{-1} \partial_r u_{\varepsilon}) \\
& \quad - 2 \varepsilon (1+ \varepsilon \rho ) \Big[ \rho_\theta r^{-1} \partial_\theta \partial_r u_{\varepsilon} + \Big(\sin^{-1}(\theta)\partial_\phi \rho \Big) \Big( r^{-1} \sin^{-1}(\theta) \partial_\phi \Big) \partial_r u_{\varepsilon}  \Big]  \\
& \quad + 2 \varepsilon^2 \Big[ (\rho_\theta^2 +  \sin^{-2}(\theta) \rho_\phi^2 \Big] (r^{-1} \partial_r u_{\varepsilon})
+ \varepsilon^2 \Big( \sin^{-2}(\theta) \rho_\phi^2 \Big)  \partial_r^2 u_{\varepsilon} \\
& = 0.
\end{align*}

Defining the operators: 
\begin{align*}
\Delta_S u &= \sin^{-1}(\theta) \partial_\theta \left( \sin(\theta) \partial_\theta u \right) + \sin^{-2}(\theta) \partial_\phi^2 u \\
L_1u & = 2\rho \Delta u + (\Delta_S\rho)r^{-1}\partial_r u -2\left(r^{-1}\rho_\theta \partial_\theta \partial_r u + \frac{\rho_\phi}{\sin\theta}\frac{\partial_\phi \partial_r u}{r\sin\theta}\right) \\
L_2u & = \rho^2 \Delta u + (\rho\Delta_S\rho)r^{-1}\partial_r u -2\rho \left(r^{-1}\rho_\theta \partial_\theta \partial_r u + \frac{\rho_\phi}{\sin\theta}\frac{\partial_\phi \partial_r u}{r\sin\theta}\right) \\
\nonumber 
& \quad + 2\left( (\rho^2_\theta + \frac{\rho^2_\phi}{\sin^2\theta} )(r^{-1}\partial_r u) +\frac{\rho^2_\phi}{\sin^2\theta} \partial^2_r u \right)
\end{align*}
The function $u_\varepsilon$ satisfies
\begin{subequations}
\label{Ueqn}
\begin{align}
\label{Ueqna}
& \Delta u_\varepsilon  = \varepsilon L_1 u_\varepsilon + \varepsilon^2 L_2 u_\varepsilon && \textrm{in } B\\
\label{Ueqnb}
&  u_\varepsilon (1,\theta,\phi) = \xi(\theta,\phi) && \textrm{on } S^2.  
\end{align}
\end{subequations}

\begin{lem}
\label{lem31}
For $s\in \mathbb N$, there is a constant $K_0 > 0$ such that for any $F \in H^{s-1}(B)$ and $\xi \in H^{s+ \frac 1 2}(S^2)$, the solution of 
\begin{align*}
& \Delta w(r,\theta,\phi) = F(r,\theta,\phi) && (r,\theta,\phi) \in B \\
& w(1,\theta,\phi) = \xi(\theta,\phi) 
\end{align*}
satisfies
\[
\| w \|_{H^{s+1}(B)} \leq K_0 \left( \| F \|_{H^{s-1}(B)}  + \| \xi \|_{H^{s+ \frac 1 2}(S^2)} \right). 
\]
\end{lem}
\begin{proof}
We will prove the result for $s=0$. Since $\xi \in H^{\frac 1 2 }(S^2)$, we have the spherical harmonic transform, 
\[
\xi(\theta,\phi) =  \sum_{\ell=0}^\infty \sum_{m=-\ell}^{\ell}\hat{\xi}(\ell,m) Y_\ell^m(\theta,\phi), 
\qquad \textrm{where} \ \ 
\hat{\xi}(\ell, m) = \int_0^{2 \pi} \int_0^\pi \xi(\theta,\phi) \overline{Y_\ell^m(\theta,\phi)}\sin\theta d \theta d \phi.
\]
 and $Y_\ell^m$ are the (complex) spherical harmonics. Set $v = w- \Phi$, where $\Phi$ solves
\begin{align*}
\Delta \Phi & = 0 && \textrm{in } B\\
\Phi & = \xi && \textrm{on } S^2.
\end{align*}
Then v satisfies
\begin{subequations}
\label{Lem21veqn}
\begin{align}
\Delta v& = F && \textrm{in } B\\
v& = 0 && \textrm{on } S^2.
\end{align}
\end{subequations}
We have
$$\Phi(r, \theta, \phi) = \sum_{\ell=0}^\infty \sum_{m=-\ell}^{\ell}r^\ell \hat{\xi}(\ell,m) Y_\ell^m(\theta,\phi),$$
and defining
$$\hat{\Phi}(r,\ell,m) = r^\ell \int_0^{2 \pi} \int_0^\pi \Phi(1, \theta,\phi) \overline{Y_\ell^m(\theta,\phi)} \ \sin\theta d \theta d \phi$$
we see that 
$$
\hat{\Phi}(r,\ell,m) = r^\ell \hat{\xi}(\ell,m).
$$
We thus calculate:
\begin{subequations} \label{e:bndPhi3d}
\begin{align}
\| \Phi \|^2_{H^1(B)} 
&= \sum_{\ell, m} \Big[ \iiint_B  \left( r^{2\ell} |\hat{\xi}(\ell,m)|^2 |Y_\ell^m(\theta,\phi)|^2 +  |\nabla r^\ell \hat{\xi}(\ell,m)Y_\ell^m(\theta,\phi)|^2 \right) \ r^2 \sin\theta d\theta d\phi dr \\
&= \sum_{\ell, m} \iiint_B |\hat{\xi}(\ell,m)|^2 |Y_\ell^m(\theta,\phi)|^2 \left( r^{2\ell} (1+\ell(\ell+1)) +\ell^2r^{2\ell-2} \right) \ r^2\sin\theta d\theta d\phi dr \\
&=\sum_{\ell, m}|\hat{\xi}(\ell,m)|^2 \left (\frac{\ell^2 + \ell + 1}{2\ell +3} + \frac{\ell^2}{2\ell+1}\right) \\
& \leq C_{\frac 1 2} \| \xi \|^2_{H^{1/2}(S^2)}
\end{align}
\end{subequations}
for some constant $C_{\frac 1 2} > 0$. 

Multiplying \eqref{Lem21veqn} by $\overline v$ and integrating by parts yields
\[ 
\| v \|_{H_0^1(B)}^2 = \int_B \nabla v \cdot \nabla \overline{v} = - \int_B F \overline{v} .
\]
By the duality of $H_0^1(B)$ and $H^{-1}(B)$, we have
$\int_B | F \overline{v}|  \leq \| F \|_{H^{-1}(B)} \| v \|_{H_0^1(B)}$. 
Since $v \in H_0^1(B)$, by the Poincar\'e inequality, there exists a constant $C_B$ such that $ \| v \|_{H^1(B)} \leq  C_B \| v \|_{H^1_0(B)}$ and we conclude that 
\begin{equation} \label{e:bndv3d}
\| v \|_{H^1(B)} \leq  C_B \| v \|_{H_0^1(B)} \leq C_B \| F \|_{H^{-1}(B)}.
\end{equation}

Using the decomposition $w = v + \Phi$ and using \eqref{e:bndPhi3d} and \eqref{e:bndv3d}, we obtain
\[ 
\| w \|_{H^1(B)} \leq \| v \|_{H^1(B)} + \| \Phi \|_{H^1(B)} 
\leq C_B \| F \|_{H^{-1}(B)} + C_{\frac 1 2} \| \xi \|_{H^{\frac 1 2}(S^2)}.
\]
Taking $K_0 = \max\{ C_{\frac 1 2}, C_B \}$ yields the desired result for $s=0$. The proof for $s\geq 1$ is similar. 
\end{proof} 

Let us make the ansatz
\begin{align}
\label{upwrseries}
u_\varepsilon(r,\theta,\phi) = \sum_{n=0}^\infty \varepsilon^n u_n(r,\theta,\phi).
\end{align}
Then, by \eqref{Ueqn}, we have the recursive formula
\begin{subequations}
\begin{align}
\Delta u_n & = L_1 u_{n-1} +  L_2 u_{n-2} && \textrm{in } B \\
u_n & = 
	\begin{cases}
	\xi & \text{if } n=0 \\
	0 & \text{if } n > 0
	\end{cases} && \textrm{on } S^2. 
\end{align}
\end{subequations}

\begin{lem}
\label{lem32}
Let $s\in \mathbb N$ and let $\rho \in C^{s+2}(S^2)$.  Assume that $K_1$ and $A$ are constants so that 
$$
\| u_n \|_{H^{s+2}(B)} \leq K_1 A^n
\qquad \qquad \textrm{for all } \ \  n < N.
$$
If $A >|\rho|_{C^{s+2}}$, 
then there exists a $C_0$ such that 
\begin{align*}
\| L_1 u_{N-1} \|_{H^s(B)} &\leq K_1 |\rho|_{C^{s+2}} C_0 A^{N-1} \\
\| L_2 u_{N-2} \|_{H^s(B)} &\leq K_1 |\rho|_{C^{s+2}} C_0 A^{N-1}.
\end{align*}
\end{lem}
\begin{proof}
Using the triangle inequality and \eqref{e:Nicholls}, we calculate:
\begin{align*}
\| L_1u_{N-1} \|_{H^s(B)} & \leq  \| 2\rho \Delta u_{N-1}\|_{H^s(B)} + \| (\Delta_{S^2}\rho)r^{-1}\partial_r u_{N-1}\|_{H^s(B)} + 2 \| r^{-1}\rho_\theta \partial_\theta \partial_r u_{N-1}\|_{H^s(B)}\\
& \quad + 2 \left\| \frac{\rho_\phi}{\sin\theta}\cdot \frac{\partial_\phi \partial_r u_{N-1}}{r\sin\theta} \right\|_{H^s(B)} \\
& \leq M(s) \Big( 2|\rho|_{C^s}\|u_{N-1}\|_{H^{s+2}(B)} + |\rho|_{C^{s+2}}\|u_{N-1}\|_{H^{s+2}(B)} + 4|\rho|_{C^{s+1}}\|u_{N-1}\|_{H^{s+2}(B)} \Big) \\
& \leq 7M(s) |\rho|_{C^{s+2}}K_1A^{N-1}.
\end{align*}
In the second inequality, we have also used that all operators acting on $u_{N-1}$ are second order. 
We similarly estimate $\|L_2 u_{N-2}\|_{H^s(B)}$:
\begin{align*}
\| L_2u_{N-2} \|_{H^s(B)} & \leq  \| \rho^2 \Delta u_{N-2} \|_{H^s(B)} + \| (\rho\Delta_{S^2}\rho)r^{-1}\partial_r u_{N-2} \|_{H^s(B)} + 2\| \rho\rho_\theta  r^{-1} \partial_\theta \partial_r u_{N-2} \|_{H^s(B)}\\
&  \quad + \left\|\frac{\rho \rho_\phi}{\sin\theta}\cdot \frac{\partial_\phi \partial_r u_{N-2}}{r\sin\theta} \right\|_{H^s(B)} 
+ 2\left\| \rho^2_\theta r^{-1}\partial_r u_{N-2} \right\|_{H^s(B)} \\
& \quad + 2 \left\|\frac{\rho^2_\phi}{\sin^2\theta}r^{-1}\partial_r u_{N-2} \right\|_{H^s(B)}
+ 2 \left\| \frac{\rho^2_\phi}{\sin^2\theta} \partial^2_r u_{N-2} \right\|_{H^s(B)} \\
& \leq M(s) \Big( |\rho|^2_{C^2}\|u_{N-2} \|_{H^{s+2}(B)} + |\rho|_{C^s} |\rho|_{C^{s+2}}\|u_{N-2} \|_{H^{s+2}(B)}\\
& \quad + 4  |\rho|_{C^s} |\rho|_{C^{s+1}}\|u_{N-2} \|_{H^{s+2}(B)}  + 6|\rho|_{C^{s+1}}\|u_{N-2} \|_{H^{s+2}(B)} \Big) \\
& \leq 12M(s) |\rho|^2_{C^{s+2}} K_1 A^{N-2} \\
& \leq 12M(s) |\rho|_{C^{s+2}} K_1 A^{N-1}.
\end{align*}
Taking $C_0 = 12 M(s)$ completes the proof.
\end{proof}

The following theorem justifies the convergence of \eqref{upwrseries} for suitably small $\varepsilon > 0$.

\begin{thm} \label{e:Thm33}
Given $s \in \mathbb N$, if $\rho \in C^{s+2}(S^2)$ and $\xi \in H^{s+\frac 3 2}(S^2)$, there exists constants $C_0$ and $K_0$ and a unique solution $u_\varepsilon$ of \eqref{Ueqn} satisfying \eqref{upwrseries} such that
\begin{equation}
\label{u-bound3d}
\|u_n \|_{H^{s+2}(B)} \leq K_0 \| \xi \|_{H^{s+ \frac 3 2}(S^2)} A^n
\end{equation}
for any $A > 2 K_0 C_0 | \rho|_{C^{s+2}}$. 
\end{thm}
\begin{proof}
We proceed by induction.  For $n=0$, we use Lemma \ref{lem31} to we see
\begin{equation*}
     \| u_0 \|_{H^{s+2}} \leq K_0 \| \xi \|_{H^{s+ \frac 3 2}} A^0,
\end{equation*}
as desired to show \eqref{u-bound3d}. We now define $K_1=K_0\| \xi \|_{H^{s+ \frac{3}{2}}(S^2)}$ for the remainder of the proof to be used in Lemma \ref{lem32}. 

Suppose inequality \eqref{u-bound3d} holds for $n<N$.  Then by Lemma \ref{lem31},
\begin{equation*}
    \| u_N \|_{H^{s+2}} \leq K_0 \left( \| L_1 u_{N-1}\|_{H^s} + \| L_2 u_{N-2}\|_{H^s} \right). 
\end{equation*}
By Lemma \ref{lem32}, we may bound $\|L_1 u_{N-1}\|_{H^s}$ and $\|L_2 u_{N-2}\|_{H^s}$ so that
\begin{align*}
    \| u_N \|_{H^{s+2}} & \leq 2K_0 K_1 C_0 |\rho|_{C^{s+2}} A^{N-1} \\
    & = 2 K_0^2 \| \xi \|_{H^{s+ \frac 3 2}(S^2)}C_0 |\rho|_{C^{s+2}} A^{N-1} \\
    & \leq K_0 \| \xi \|_{H^{s+ \frac 3 2}(S^2)} A^N
\end{align*}
provided $A> 2K_0 C_0 |\rho|_{C^{s+2}}$.
\end{proof}

\subsection{Proof of Theorem~\ref{t:Thm1} in three dimensions: Analyticity of the Dirichlet to Neumann operator} 
\label{s:DNO-anal-3D}

Denote the Dirichlet-to-Neumann operator $G_{\rho, \varepsilon}\colon H^{s+1/2}(\partial \Omega_\varepsilon) \rightarrow H^{s-1/2}(\partial \Omega_\varepsilon)$ which is defined
\begin{align*}
G_{\rho, \varepsilon}\xi = \vec{n}_\varepsilon \cdot \nabla v
\end{align*}
where $v$ satisfies \eqref{e:Laplace3d} and 
\begin{align*}
\vec{n}_\varepsilon & = 
\left( (1+2\varepsilon \rho)^2 + \varepsilon^2 \rho_\theta^2 + \varepsilon^2 \frac{\rho^2_\phi}{\sin^2\theta}\hat{\phi}\right)^{-1/2}
\left( (1+\varepsilon \rho) \hat{r} - \varepsilon \rho_\theta \hat{\theta} - \varepsilon \frac{\rho_\phi}{\sin\theta}\hat{\phi}\right)  \\
& = M_\rho(\varepsilon) \left( (1+\varepsilon \rho) \hat{r} - \varepsilon \rho_\theta \hat{\theta} - \varepsilon \frac{\rho_\phi}{\sin\theta}\hat{\phi}\right )
\end{align*}
is the unit-length normal vector on $\partial \Omega_\varepsilon$.  Here the spherical coordinate vectors $\hat{r}, \hat{\theta}, \hat{\phi}$ are given by
$$
\hat r = \begin{pmatrix}
\sin(\theta) \cos (\phi) \\ \sin(\theta) \sin(\phi) \\ \cos(\theta)
\end{pmatrix}, 
\qquad
\hat \theta = \begin{pmatrix}
\cos(\theta) \cos (\phi) \\ \cos(\theta) \sin(\phi) \\ -\sin(\theta)
\end{pmatrix}, 
\quad \textrm{and} \quad
\hat \phi = \begin{pmatrix}
- \sin(\phi) \\ \cos(\phi) \\ 0
\end{pmatrix}. 
$$
Making the change of variables in \eqref{e:ChangeVariables3d}, we obtain 
\begin{subequations}
\begin{align}
\label{DnoRep}
G(1 + \varepsilon \rho(\theta,\phi) ) \xi & = M_\rho(\varepsilon)\left[ (1+ \frac{\varepsilon^2}{(1+\varepsilon\rho)^2}\left (\rho^2_\theta + \frac{\rho^2_\phi}{\sin^2\theta}\right )\partial_r u - \frac{\varepsilon\rho_\theta}{1+\varepsilon \rho} \partial_\theta u - \frac{\varepsilon\rho_\phi}{(1+\varepsilon\rho)\sin^2\theta}\partial_\phi u \right] \\
& = M_\rho(\varepsilon) \hat{G}_{\rho, \varepsilon} \xi
\end{align}
\end{subequations}
Since $M_\rho(\varepsilon)$ is clearly analytic near $\varepsilon = 0$, we need only show the analyticity of $\hat{G}_{\rho, \varepsilon}$ near $\varepsilon = 0$. to verify that $G_{\rho, \varepsilon}$ is analytic as well.  Note that $\hat{G}_{\rho, \varepsilon}$ satisfies
\begin{align}
\label{DnoHatRep}
\hat{G}_{\rho, \varepsilon}\xi & =\partial_r u + \varepsilon \left( 2\rho \partial_r u  - \rho_\theta \partial_\theta u - \frac{\rho_\phi}{\sin^2\theta}\partial_\phi u -2\rho \hat{G}_{\rho, \varepsilon}\xi\right) \\
\nonumber
& \quad + \varepsilon^2 \left[ \left(\rho^2 + \rho^2_\theta+ \frac{\rho^2_\phi}{\sin^2\theta}\right) \partial_r u - \rho \rho_\theta \partial_\theta u - \frac{\rho \rho_\phi}{\sin^2\theta}\partial_\phi u -  \rho^2\hat{G}_{\rho, \varepsilon}\xi \right] 
\end{align}

We now make a power series ansatz for the non-normalized DNO $\hat{G}_{\rho, \varepsilon}$:
\begin{align}
\label{DNOpwrseries}
\hat{G}_{\rho, \varepsilon}\xi  = \sum_{n=0}^\infty \varepsilon^n \hat{G}_{\rho, n}\xi
\end{align}
for $\xi \in H^{s+\frac 1 2}(S^2)$ and $s \in \mathbb{N}$. By \eqref{upwrseries} and \eqref{DnoHatRep}, we obtain the recursive relationship, 
\begin{align}
\label{DnoRec}
\hat{G}_{\rho, n} \xi & = \partial_r u_n + 2\rho \partial_r u_{n-1} - \rho_\theta \partial_\theta u_{n-1} - \frac{\rho_\phi}{\sin^2\theta}\partial_\phi u_{n-1} + \left(\rho^2 + \rho^2_\theta + \frac{\rho^2_\phi}{\sin^2\theta} \right)\partial_r u_{n-2}\\
\nonumber
& \quad - \rho \rho_\theta \partial_\theta u_{n-2} -\frac{\rho \rho_\phi}{\sin^2\theta} \partial_\phi u_{n-2} - 2\rho \hat{G}_{\rho, n-1} \xi - \rho^2\hat{G}_{\rho, n-2} \xi 
\end{align}

The following theorem proves Theorem \ref{t:Thm1} in three dimensions and justifies the convergence of \eqref{DNOpwrseries} for suitably small $\varepsilon > 0$.

\begin{thm} \label{t:ConvDNO3d}
Let $\xi \in H^{s+3/2}(S^2)$.  Then
\begin{align}
\label{DNOrecursivebnd3d}
\|\hat{G}_{\rho, n}\xi \|_{H^{s+1/2}(S^2)} \leq K_1 \| \xi \|_{H^{s+3/2}(S^2)}A^n
\end{align}
for $A> C |\rho|_{C^{s+2}}$
\end{thm}
\begin{proof}
We will proceed via induction. First, we show \eqref{DNOrecursivebnd3d} fo $n=0$:
\begin{align*}
    \|\hat{G}_{\rho, 0}\xi \|_{H^{s+1/2}(S^2)} &  \leq \| \partial_r u_0\|_{H^{s+1/2}(S^2)} \leq C_1 \| \partial_r u_0\|_{H^{s+1}(B)} \\
    & \leq C_1\|u_0\|_{H^{s+2}(B)} \leq C_1 K_0 \| \xi \|_{H^{s+3/2}(S^2)}
\end{align*}
In the second inequality of the first line, we have used the standard Trace theorem, while Theorem \ref{e:Thm33} is used in the second line.
Now suppose that \eqref{DNOrecursivebnd3d} holds for $n<N$.  Then we have the following estimate:
\begin{align*}
\| \hat{G}_{\rho, N} \xi \|_{s+\frac 1 2 } 
& \leq \left\| \partial_r u_N \right\|_{s+\frac 1 2 } 
+ 2 \left\| \rho \partial_r u_{N-1} \right\|_{s+\frac 1 2 }  
+ \left\| \rho_\theta \partial_\theta u_{N-1} \right\|_{s+\frac 1 2 }  
+ \left\| \frac{\rho_\phi}{\sin^2\theta}\partial_\phi u_{N-1} \right\|_{s+\frac 1 2 }  \\
& \quad + \left\| \rho^2 \partial_r u_{N-2} \right\|_{s+\frac 1 2 } 
+ \left\| \rho^2_\theta \partial_r u_{N-2} \right\|_{s+\frac 1 2 } 
+ \left\| \frac{\rho^2_\phi}{\sin^2\theta}\partial_r u_{N-2} \right\|_{s+\frac 1 2 }
+ \left\| \rho \rho_\theta \partial_\theta u_{N-2} \right\|_{s+\frac 1 2 } \\
&  \quad +  \left\| \frac{\rho \rho_\phi}{\sin^2\theta} \partial_\phi u_{N-2} \right\|_{s+\frac 1 2 } 
+ 2 \left\| \rho \hat{G}_{\rho, N-1} \xi \right\|_{s+\frac 1 2 } 
+ \left\| \rho^2\hat{G}_{\rho, N-2} \xi \right\|_{s+\frac 1 2 }\\
& \leq C_1K_0 \left\| \xi \right\|_{s+\frac 3 2 }A^N 
+ M(s)C_1K_0 \Big( 2 |\rho|_{C^{s+\frac 1 2 +\delta}}  \left\| \xi \right\|_{s+\frac 3 2 }A^{N-1}  \\
& \quad + |\rho|_{C^{s+\frac 3 2 +\delta}} \| \xi \|_{s+\frac 3 2 }A^{N-1}  +|\rho|_{C^{s+\frac 3 2 +\delta}}\| \xi \|_{s+\frac 3 2 }A^{N-1}  \\
& \quad + |\rho|^2_{C^{s+\frac 1 2 +\delta}} \| \xi \|_{s+\frac 3 2 }A^{N-2} 
+ |\rho|^2_{C^{s+\frac 3 2 +\delta}} \| \xi \|_{s+\frac 3 2 }A^{N-2} 
+ |\rho|^2_{C^{s+\frac 3 2 +\delta}} \| \xi \|_{s+\frac 3 2 }A^{N-2} \\
& \quad + |\rho|_{C^{s+\frac 1 2 +\delta}}|\rho|_{C^{s+\frac 3 2 +\delta}} \|\xi\|_{s+\frac 3 2 }A^{N-2}  
+  |\rho|_{C^{s+\frac 1 2 +\delta}}|\rho|_{C^{s+\frac 3 2 +\delta}} \| \xi \|_{s+\frac 3 2 }A^{N-2} \\
& \quad + K_1 |\rho|_{C^{s+\frac 1 2 +\delta}} \| \xi \|_{s+\frac 3 2 }A^{N-1}+ K_1 |\rho|^2_{C^{s+\frac 1 2 +\delta}} \| \xi \|_{s+\frac 3 2 }A^{N-2} \Big)\\
& \leq K_1 \| \xi \|_{s+ \frac 3 2 } A^N, 
\end{align*}
for $K_1 = \max \{ 2C_1K_0, 2C_1K_0M(s) \}$ and $A > C|\rho|_{C^{s+2}}$.
\end{proof}

\section{Proof of Corollary~\ref{c:Cor1}: Analyticity of the Steklov eigenvalues}
\label{s:Steklov-anal}
We now have all of the ingredients to prove Corollary \ref{c:Cor1}. 

\begin{proof}[Proof of Corollary \ref{c:Cor1}.]
Theorems \ref{t:ConvDNO2d} and \ref{t:ConvDNO3d}  show that the expansion of the non-normalized DNO $\hat G_{\rho, \varepsilon}$ given in \eqref{DNOanalytic} and \eqref{DNOpwrseries} is uniformly convergent for small $\varepsilon$. 
It follows that $G_{\rho, \varepsilon} \colon H^{s+\frac 3 2}(S^1) \to H^{s + \frac 1 2}(S^1)$ is analytic for small $\varepsilon$.  The DNO operator $ G_{\rho, \varepsilon}\colon L^2(S^{d-1}) \to L^2(S^{d-1})$ is self-adjoint \cite{ArendtDNO}, hence closed. 
The result now follows from \cite[Ch. 7, Thm 1.8, p. 370]{Kato_1966}. 
\end{proof}

\subsection*{Acknowledgments} We would like to thank Nilima Nigam and Fadil Santosa for helpful discussions. B. Osting is partially supported by NSF DMS 16-19755 and 17-52202.

\bibliographystyle{plainurl}
\bibliography{SteklovAsymptotics}

\end{document}